\documentclass{amsart}

\usepackage{fancyhdr}

\renewcommand{\topmargin}{22.5pt}

\usepackage{hyperref}
\usepackage{enumitem}

\usepackage{accents}

\newtheorem{theorem}{Theorem}[section]
\newtheorem{lemma}[theorem]{Lemma}

\newtheorem{proposition}[theorem]{Proposition}

\theoremstyle{definition}
\newtheorem{definition}[theorem]{Definition}

\numberwithin{equation}{section}

\providecommand{\abs}[1]{\left\lvert{#1}\right\rvert}

\providecommand{\norm}[1]{\left\lVert{#1}\right\rVert}

\newcommand{\dual}[2]{\left\langle{#1}, \, {#2}\right\rangle}

\newcommand{\lr}[1]{\left({#1}\right)}
\newcommand{\lra}[1]{\left\lbrack{#1}\right\rbrack}
\newcommand{\lrb}[1]{\left\lbrace{#1}\right\rbrace}

\newcommand{\calka}[2][x]{\int_{\Omega} {#2} \, \mathrm{d}{#1}}
\newcommand{\tcalka}[2][x]{\int_{\mathbb{T}^3} {#2} \, \mathrm{d}{#1}}

\newcommand{\calkat}[2][t]{\int_{0}^{T} {#2} \, \mathrm{d}{#1}}

\newcommand{\dv}{\operatorname{div}}
\newcommand{\supp}{\mathop{\rm supp}}

\newcommand{\subscript}[2]{$#1 _ #2$}

\begin{document}

\title{Energy Equality for the $3$D Critical Convective Brinkman--Forchheimer Equations}

\author{Karol W. Hajduk}
\address{Mathematics Institute, Zeeman Building, University of Warwick, 
Coventry, CV4 7AL, United Kingdom}
\email{k.w.hajduk@warwick.ac.uk}
\thanks{KWH was supported by an EPSRC Standard DTG EP/M506679/1 and by the Warwick Mathematics Institute.}

\author{James C. Robinson}
\address{Mathematics Institute, Zeeman Building, University of Warwick, 
Coventry, CV4 7AL, United Kingdom}
\email{j.c.robinson@warwick.ac.uk}
\thanks{JCR was supported in part by an EPSRC Leadership Fellowship EP/G007470/1.}

\subjclass[2000]{Primary 35Q35, 76S05; Secondary 76D03}

\date{August 10, 2017.}


\keywords{Brinkman--Forchheimer, Tamed Navier--Stokes, Energy equality, Critical exponent, Global strong solutions, Strong global attractor}

\begin{abstract}
In this paper we give a simple proof of the existence of global-in-time smooth solutions for the convective Brinkman--Forchheimer equations (also called in the literature the tamed Navier--Stokes equations)
$$ \partial_tu -\mu\Delta u + (u \cdot \nabla)u + \nabla p + \alpha u + \beta\abs{u}^{r - 1}u = 0 $$
on a $3$D periodic domain, 
for values of the absorption exponent $r$ larger than $3$. Furthermore, we prove that global, regular solutions exist also for the~critical value of exponent $r = 3$, provided that the coefficients satisfy the relation $4\mu\beta \geq 1$. Additionally, we show that in the critical case every weak solution verifies the energy equality and hence is continuous into the phase space $L^2$. As an application of this result we prove the existence of a strong global attractor, using the theory of evolutionary systems developed by Cheskidov.
\end{abstract}

\maketitle

\section{Introduction}

\makeatletter
    \renewcommand{\theequation}{{\thesection}.\@arabic\c@equation}
    \renewcommand{\thetheorem}{{\thesection}.\@arabic\c@theorem}
    \renewcommand{\thedefinition}{{\thesection}.\@arabic\c@definition}
    \renewcommand{\thecorollary}{{\thesection}.\@arabic\c@corollary}
    \renewcommand{\theproposition}{{\thesection}.\@arabic\c@proposition}
\makeatother

In this paper we consider both weak and strong solutions of the three-dimensional incompressible convective Brinkman--Forchheimer equations (CBF)
\begin{equation} \label{cbfequation}
\left\{\begin{split}
\partial_tu -\mu\Delta u + (u \cdot \nabla)u + \nabla p + \alpha u + \beta\abs{u}^{r - 1}u &= 0, \\
\mbox{div\,} u &= 0, \\
u(x, 0) &= u_0(x),
\end{split}\right.
\end{equation}
where $u(x, t) = (u_1, u_2, u_3)$ is the velocity field and the scalar function $p(x, t)$ is the~pressure. The constant $\mu$ denotes the positive Brinkman coefficient (effective viscosity). The positive constants $\alpha$ and $\beta$ denote respectively the Darcy (permeability of porous medium) and Forchheimer (proportional to the porosity of the~material) coefficients. The exponent $r$ can be greater than or equal to $1$. The domain on which we consider problem $(\ref{cbfequation})$ is a three-dimensional torus $\mathbb{T}^3 = [0, 2\pi]^3$, with periodic boundary conditions and zero mean-value constraint for the functions, i.e.\ $\int_{\mathbb{T}^3}{u(x, t) \, \mathrm{d}x} = 0$. Very similar arguments should work to prove most of the results given here (except perhaps the energy equality result) also for the Cauchy problem, i.e.\ in the case when the domain is the whole space $\mathbb{R}^3$. However, when $\Omega \subset \mathbb{R}^3$ is an~open, bounded domain with Dirichlet boundary conditions $u|_{\partial\Omega} = 0$, not all the~results proved here are that straightforward. In particular, one has to be very careful in choosing an approximation in the proof of Theorem $\ref{twr:ee}$, and we will address this problem in a~future paper.

The CBF equations $(\ref{cbfequation})$ describe the motion of incompressible fluid flows in a saturated porous medium. While the motivation for introducing an absorption term $\abs{u}^{r - 1}u$ is purely mathematical, this model is used in connection with some real world phenomena, e.g.\ in the theory of non-Newtonian fluids as well as tidal dynamics (see \cite{KalantarovZelik}, \cite{RocknerZhang}, \cite{YouZhao} and references therein). However, its applicability is believed to be limited to flows when the velocities are sufficiently high and porosities are not too small, i.e.\ when the Darcy law for a porous medium no longer applies (for more details see \cite{Nield2}, \cite{Nield3} and the discussion in \cite{MarkowichTiti}).

In this paper we use the standard notation for the vector-valued function spaces which often appear in the theory of fluid dynamics. For an arbitrary domain $\Omega \subseteq \mathbb{R}^n$ we define:
\begin{align} \nonumber
C_{0}^{\infty}(\Omega) &:= \lrb{\varphi \in C^{\infty}(\Omega) : \supp{\varphi} \ \mbox{is compact}}, \\ \nonumber
\mathcal{D}_{\sigma}(\Omega) &:= \lrb{\varphi \in C_{0}^{\infty}(\Omega) : \dv{\varphi} = 0}, \\ \nonumber
X_q(\Omega) &:= \mbox{closure of} \ \mathcal{D}_{\sigma}(\Omega) \ \mbox{in the Lebesgue space} \ L^q(\Omega), \\ \nonumber
V^s(\Omega) &:= \mbox{closure of} \ \mathcal{D}_{\sigma}(\Omega) \ \mbox{in the Sobolev space} \ W^{s, 2}(\Omega) \quad \mathrm{for} \quad s > 0.
\end{align}
The space of divergence-free test functions in the space-time domain is denoted by
$$ \mathcal{D}_{\sigma}({\Omega_T}) := \lrb{\varphi \in C_{0}^{\infty}({\Omega_T}) : \dv{\varphi(\cdot, t)} = 0}, $$
where $\Omega_T := \Omega \times [0, T)$ for $T > 0$. Note that $\varphi(x, T) = 0$ for all $\varphi \in \mathcal{D}_{\sigma}({\Omega_T})$.

We denote the Hilbert space $X_2(\Omega)$ by $H$, $V^1(\Omega)$ by $V$ and $V^s(\Omega)$ by $V^s$ for $s \neq 1$. The space $H$ is endowed with the inner product induced by $L^2(\Omega)$. We denote it by $\dual{\cdot}{\cdot}$, and the~corresponding norm is denoted by $\norm{\cdot}$. 

For simplicity, we consider here only the case when there are no external forces acting on the fluid, i.e.\ the right-hand side of $(\ref{cbfequation})$ is equal to zero. All the results proved in this paper hold as well for non-zero forces $f(x, t)$ and the proofs can be carried out in the similar way with some natural assumptions on the regularity of the function $f$. For example, for Theorem $\ref{twr:ee}$ one can assume that $f \in L^1({0, T; H})$.

We consider here primarily the CBF equations when the exponent $r = 3$
\begin{equation} \label{criticalcbf}
\partial_tu -\mu\Delta u + (u \cdot \nabla)u + \nabla p + \alpha u + \beta\abs{u}^{2}u = 0.
\end{equation}
The critical homogenous CBF equations $(\ref{criticalcbf})$ have the same scaling as the Navier--Stokes equations (NSE) only when the permeability coefficient $\alpha $ is equal to zero. In this case the model is sometimes referred to in the literature as the Navier--Stokes equations modified by an absorption term (see e.g.\ \cite{Oliveira}) or the tamed Navier--Stokes equations (see e.g. \cite{RocknerZhang}). We lose the scale-invariance property for other values of the parameters $r$ and $\alpha$. This follows at once from the~following simple proposition.
\begin{proposition}
Let $\Omega$ be the whole space $\mathbb{R}^n$.
Let $u_{\lambda}$ be the usual parabolic rescaling of the velocity field $u$:
$$ u_{\lambda}(x, t) := \lambda u(\lambda x, \lambda^2 t) \quad \mbox{for} \quad \lambda > 0, $$
and let $p_{\lambda}$ be the usual rescaling of the pressure function $p$:
$$ p_{\lambda}(x, t) := \lambda^2 p(\lambda x, \lambda^2 t) \quad \mbox{for} \quad \lambda > 0. $$
If $u$ and $p$ solve the CBF equations $(\ref{cbfequation})$, then the rescaled functions $u_{\lambda}, p_{\lambda}$ satisfy
\begin{equation} \nonumber
\partial_t u_{\lambda} - \mu\Delta u_{\lambda} + (u_{\lambda} \cdot \nabla)u_{\lambda} + \nabla p_{\lambda} + \lambda^{2}\alpha u_{\lambda} + \lambda^{3 - r}\beta\abs{u_{\lambda}}^{r - 1}u_{\lambda} = 0.
\end{equation}
\end{proposition}

Since the linear term $\alpha u$ poses no additional difficulties in the mathematical treatment of the problem, we set the parameter $\alpha$ to zero. In what follows we will consider only the following equations
\begin{equation} \label{alphazero}
\partial_tu -\mu\Delta u + (u \cdot \nabla)u + \nabla p + \beta\abs{u}^{r - 1}u = 0
\end{equation}
and, consequently, the critical case when $r = 3$
\begin{equation} \label{alphazerocrit}
\partial_tu -\mu\Delta u + (u \cdot \nabla)u + \nabla p + \beta\abs{u}^{2}u = 0.
\end{equation}
All proofs for the full problem $(\ref{cbfequation})$ can be easily rewritten based on the arguments given for equations $(\ref{alphazero})$.

We will use here the following definition of a weak solution.
\begin{definition} \label{defi:slabecbfr}
We will say that the function $u$ is \emph{a weak solution} on the time interval $[0, T)$ of the convective Brinkman--Forchheimer equations $(\ref{alphazero})$ with the~initial condition $u_0 \in H$, if
$$ u \in L^{\infty}({0, T; H}) \cap L^{r + 1}({0, T; X_{r + 1}}) \cap L^2({0, T; V}) $$
and
\begin{align} \nonumber
-\int_{t_0}^{t_1}{\dual{u(s)}{\partial_t\varphi(s)}\, \mathrm{d}s} + \mu\int_{t_0}^{t_1}{\dual{\nabla u(s)}{\nabla \varphi(s)}\, \mathrm{d}s} + \int_{t_0}^{t_1}{\dual{(u(s) \cdot \nabla)u(s)}{\varphi(s)}\, \mathrm{d}s} \\ \nonumber
+ \beta\int_{t_0}^{t_1}{\dual{\abs{u(s)}^{r - 1}u(s)}{\varphi(s)}\, \mathrm{d}s} = -\dual{u(t_1)}{\varphi(t_1)} + \dual{u(t_0)}{\varphi(t_0)}
\end{align}
for all $0 \leq t_0 < t_1 < T$ and all test functions $\varphi \in \mathcal{D}_{\sigma}({\Omega_T})$.
\end{definition}
A function $u$ is called \textit{a global weak solution} if it is a weak solution for all $T > 0$.

\begin{definition}
\textit{A Leray--Hopf weak solution} of the convective Brinkman--\linebreak Forchheimer equations $(\ref{alphazero})$ with the initial condition $u_0 \in H$ is a weak solution satisfying the following \emph{strong energy inequality}:
\begin{align} \label{energyineq}
\norm{u(t_1)}^2 + 2\mu\int_{t_0}^{t_1}{\norm{\nabla u(s)}^2 \, \mathrm{d}s} + 2\beta\int_{t_0}^{t_1}{\norm{u(s)}_{L^{r + 1}(\Omega)}^{r + 1} \, \mathrm{d}s} \leq \norm{u(t_0)}^2
\end{align}
for almost all initial times $t_0 \in [0, T)$, including zero, and all $t_1 \in (t_0, T)$.
\end{definition}

It is known that for every $u_0 \in H$ there exists at least one global Leray--Hopf weak solution of $(\ref{alphazero})$. For the proof of existence of global weak solutions for the case $\alpha = 0$ see \cite{Oliveira}.

We want to recall that for the Navier--Stokes equations $(\alpha = \beta = 0)$ it is well known that for every $u_0 \in H$ there exists at least one global Leray--Hopf weak solution that satisfies the strong  energy inequality:
\begin{equation} \label{star}
\norm{u(t_1)}^2 + 2\mu\int_{t_0}^{t_1}{\norm{\nabla u(s)}^2 \, \mathrm{d}s} \leq \norm{u(t_0)}^2.
\end{equation}
This can be found in many places, e.g.\ in Galdi \cite{Galdi} or in the recent book by Robinson, Rodrigo and Sadowski~\cite{RRSbook}.
However, it is not known if all weak solutions have to verify $(\ref{star})$.~The problem of proving equality in $(\ref{star})$ for weak solutions is also open; there are only partial results in this direction. In particular, it is known that the energy equality is satisfied by more regular weak solutions for which
\begin{equation} \label{serrincond}
u \in L^{p_1}({0, T; L^{p_2}(\Omega)}), \quad \mbox{where} \quad \frac{2}{p_1} + \frac{2}{p_2} \leq 1 \quad \mbox{and} \quad p_2 \geq 4,
\end{equation}
or
$$ \quad p \in L^{2}({0, T; L^2(\Omega)}). $$
This result was established in several papers by Lions \cite{Lions}, Serrin \cite{Serrin}, Shinbrot \cite{Shinbrot} and Kukavica \cite{Kukavica}.

From now on, unless stated otherwise, we will assume that $\Omega$ is the three-dimensional torus $\mathbb{T}^3$.

We now make the observation that weak solutions of the CBF equations $(\ref{alphazerocrit})$ by definition satisfy the condition $(\ref{serrincond})$. This suggests that the energy equality holds for all weak solutions of this problem, and we prove this in the following theorem.
\begin{theorem} \label{twr:ee}
Every weak solution of $(\ref{alphazerocrit})$ with the initial condition $u_0 \in H$ satisfies \emph{the energy equality}:
\begin{align} \label{equality}
\norm{u(t_1)}^2 + 2\mu\int_{t_0}^{t_1}{\norm{\nabla u(s)}^2 \, \mathrm{d}s} + 2\beta\int_{t_0}^{t_1}{\norm{u(s)}_{{L}^{4}(\mathbb{T}^3)}^{4} \, \mathrm{d}s} = \norm{u(t_0)}^2
\end{align}
for all $0 \leq t_0 < t_1 <T$. Hence, all weak solutions are continuous functions into the~phase space $L^2$, i.e.\ $u \in C(\lra{0, T}; H)$.
\end{theorem}

To the best of our knowledge, the validity of the energy equality is not to date verified for the convective Brinkman--Forchheimer equations $(\ref{alphazero})$ for the range of exponent values $r \in [1, 3]$. For larger values of the exponent $r > 3$, it was already shown that the CBF equations enjoy existence of global-in-time strong solutions (see proof for bounded domains in \cite{KalantarovZelik}) and hence the energy equality is satisfied. Theorem $\ref{twr:ee}$ extends the energy equality to the critical case $r = 3$. A similar result was given in the~paper of Cheskidov et al.\ \cite{Friedlander} where the energy equality was proved to hold for weak solutions of the NSE in the functional space
$$ L^3({0, T; {D}({A^{5/12}})}). $$
Here ${D}({A^{5/12}})$ is the domain of the fractional power of the Stokes operator \linebreak $A:= -\mathbb{P}\Delta$, where $\mathbb{P} : L^2 \to H$ is the Leray projection (for references see \cite{ConstantinFoias}, \cite{RRSbook} or \cite{Temam}). This space corresponds to the fractional Sobolev space $H^{5/6}$. 

The result of Theorem $\ref{twr:ee}$ was stated without a proof in \cite{Oliveira} for the Navier--Stokes equations modified by an absorption term considered on a bounded domain $\Omega$ with a compact boundary $\partial\Omega$; it perhaps seems at first glance that given the results for the NSE mentioned above no new proof is required in this case. However, in fact one has to argue more carefully in order to handle
the additional nonlinear term: one cannot use the usual truncations of the Fourier series as an approximating sequence, since we have regularity of solutions in a Lebesgue space rather than in a Sobolev space. Therefore, even in the periodic case discussed here, we use more carefully truncated Fourier series to obtain our result. We adapt the proof given in Galdi \cite{Galdi}, where a~specific mollification in time is used.

\section{Preliminaries}

\makeatletter
    \renewcommand{\theequation}{{\thesection}.\@arabic\c@equation}
    \renewcommand{\thetheorem}{{\thesection}.\@arabic\c@theorem}
    \renewcommand{\thedefinition}{{\thesection}.\@arabic\c@definition}
    \renewcommand{\thecorollary}{{\thesection}.\@arabic\c@corollary}
    \renewcommand{\theproposition}{{\thesection}.\@arabic\c@proposition}
\makeatother

In the sequel we will make use of the following auxiliary lemma, whose proof consists of integration by parts and differentiation of the absolute value function (see \cite{RobinsonSadowski}).
\begin{lemma} \label{witeklemma}
For every ${r \geq 1}$, if $u \in W^{2, 2}(\Omega)$, where $\Omega$ is either the whole space $\mathbb{R}^3$ or the three-dimensional torus $\mathbb{T}^3$ with $\calka{u} = 0$, or an open, bounded domain $\Omega \subset \mathbb{R}^3$ with Dirichlet boundary conditions $u|_{\partial\Omega} = 0$, then
\begin{equation} \nonumber
\int_{\Omega} -\Delta u \cdot \abs{u}^{r - 1}u \, \mathrm{d}x \geq \int_{\Omega} \abs{\nabla u}^{2}\abs{u}^{r -1} \, \mathrm{d}x.
\end{equation}
\end{lemma}

Explicitly, the left-hand side of the above equals (integrating by parts)
\begin{align} \nonumber
\int_{\Omega} -\Delta u \cdot \abs{u}^{r - 1}u \, \mathrm{d}x &= \int_{\Omega} \abs{\nabla u}^{2}\abs{u}^{r - 1} \, \mathrm{d}x + 4\lra{\frac{r - 1}{(r + 1)^2}}\calka{\abs{\nabla{\abs{u}^{(r + 1)/2}}}^2} \\ \nonumber
&= \int_{\Omega} \abs{\nabla u}^{2}\abs{u}^{r - 1} \, \mathrm{d}x + \frac{(r - 1)}{4}\calka{\abs{u}^{r - 3}\abs{\nabla{\abs{u}^{2}}}^2}.
\end{align}
In particular, by Lemma \ref{witeklemma}, we can write for the absorption term $\abs{u}^{r - 1}u$
\begin{equation} \nonumber
0 \leq \calka{\abs{\nabla u}^{2}\abs{u}^{r -1}} \leq \calka{-\Delta u \cdot \abs{u}^{r- 1}u} \leq r\calka{\abs{\nabla u}^{2}\abs{u}^{r -1}}.
\end{equation}

We will also need another lemma from the same paper \cite{RobinsonSadowski}.
\begin{lemma} \label{witeklemmadwa}
Take $2 \leq p < 3$. Then there exists a constant $c_p > 0$ such that, for every $u$ in the Sobolev space $W^{1, p}({\mathbb{R}^3})$ we have $u \in L^{3(r + 1)}({\mathbb{R}^3})$ and
\begin{equation} \nonumber
\norm{u}_{L^{3(r + 1)}(\mathbb{R}^3)}^{r + 1} \leq c_p \int_{\mathbb{R}^3}{\abs{\nabla u}^2\abs{u}^{r - 1} \, \mathrm{d}x},
\end{equation}
where $r + 1= p/(3 - p)$. The same is true if $\Omega$ is a bounded (perhaps periodic) domain and $u \in W^{1, p}(\Omega)$ with $\calka{u} = 0$ or $u|_{\partial \Omega} = 0$.
\end{lemma}
From the proof of Lemma \ref{witeklemmadwa}, it is clear that whenever we can approximate the~function $u$ in the space $C_0^{\infty}(\Omega)$ and whenever
$$ I_r(u) := \calka{\abs{\nabla u}^2\abs{u}^{r - 1}} < \infty \quad \mbox{for} \quad r \geq 1, $$
then $u$ belongs to the space $L^{3(r + 1)}(\Omega)$.

Boundedness of the quantity $I_r(u)$ implies as well that the function $u \in W^{1, 1}(\Omega)$ belongs also to the certain type of Besov space, namely to the Nikol'ski\u \i \  space\footnote{Nikol'ski\u \i \ spaces are a particular case of the Besov spaces when one of the exponents is fixed: $\mathcal{N}^{s, p} = B^{s, p}_{\infty}$. See \cite{Simon} for more information about these spaces.} $\mathcal{N}^{2/(r + 1), r + 1}$. In particular, we have
\begin{equation} \label{maleklemma}
\norm{u}^{r + 1}_{\mathcal{N}^{2/(r + 1), r + 1}(\Omega)} \leq c \, I_r(u),
\end{equation}
where $c > 0$ is a constant depending only on $r$ and $\Omega$  (see \cite{Malek} for the details).

We introduce the Nikol'ski\u \i \ spaces $\mathcal{N}^{s, p}$ for $p \in [1, \infty)$ and $s = m + \sigma$, where $m \geq 0$ is an integer and $\sigma \in (0, 1)$, as the subspaces of the $L^p$ functions for which the following norm
\begin{align} \nonumber
\norm{u}^p_{\mathcal{N}^{s, p}(\Omega)} &:= \norm{u}^p_p + \abs{u}^p_p \\ \nonumber
&:= \norm{u}^p_p + \sum_{\abs{\alpha} = m}{\sup_{0 < \abs{h} < \delta}\calka{ \frac{ \abs{\partial^{\alpha}u(x + h) - \partial^{\alpha}u(x)}^p} {\abs{h}^{\sigma p}} } }
\end{align}
is finite. Here $\delta > 0$ is fixed. We have for any $\varepsilon \in (0, 1)$ the following embeddings (see \cite{Nikolskii})
$$ \mathcal{N}^{s, p} \hookrightarrow W^{s - \varepsilon, p} \hookrightarrow \mathcal{N}^{s - \varepsilon, p}. $$

Applying the tools described above we will show in Theorem $\ref{globalstrong}$ that strong solutions of the convective Brinkman--Forchheimer equations with $r > 3$ possess additional regularity compared to the corresponding solutions of the Navier--Stokes equations. Despite this fact, we will use in this paper the usual definition of a~strong solution.
\begin{definition}
We say that a vector field $u$ is a \emph{strong solution} of the CBF equations $(\ref{alphazero})$ if, for the initial condition $u_0 \in V$, it is a weak solution and additionally it possesses higher regularity, namely
$$ u \in L^{\infty}({0, T; V}) \cap L^2({0, T; V^2}). $$
\end{definition}
A function $u$ is called \textit{a global-in-time strong solution} if it is a strong solution for all $T > 0$.

\section{Global existence for $r > 3$}

Now we will provide a simple proof of the global-in-time existence of strong solutions for the convective Brinkman--Forchheimer equations in the case $r > 3$. This result was given in \cite{KalantarovZelik} for a broader class of nonlinearities in bounded domains $\Omega \subset \mathbb{R}^3$ and for more regular initial conditions $u_0 \in V^2(\Omega)$, where the proof was based on a~nonlinear localisation technique.
\begin{theorem} \label{globalstrong}
For every initial condition $u_0 \in V$ and for every exponent $r > 3$, there exists a global-in-time strong solution of the CBF equations $(\ref{alphazero})$. Moreover, this solution belongs to the spaces
\begin{equation} \label{spaces}
L^{r + 1}({0, T; X_{3(r + 1)}(\mathbb{T}^3)}) \quad \mbox{and} \quad L^{r + 1}({0, T; \mathcal{N}^{2/(r + 1), r + 1}(\mathbb{T}^3)})
\end{equation}
for all $T > 0$.
\end{theorem}

We present here only formal calculations which can be justified rigorously via a~Galerkin approximation argument.
\begin{proof}
Multiplying both sides of $(\ref{alphazero})$ by $-\Delta u$ and integrating over $\mathbb{T}^3$, we obtain
\begin{align} \nonumber
\frac{1}{2}\frac{\mathrm{d}}{\mathrm{d}t}\norm{\nabla u}^2 + \mu\norm{\Delta u}^2 + \beta\dual{\abs{u}^{r - 1}u}{-\Delta u} + \dual{(u \cdot \nabla)u}{-\Delta u} = 0.
\end{align}

Using Lemma \ref{witeklemma} we note that
\begin{align} \nonumber
\dual{\abs{u}^{r - 1}u}{-\Delta u} \geq \tcalka{ \abs{u}^{r - 1}\abs{\nabla u}^2 }.
\end{align}
This gives us
\begin{align} \nonumber
\frac{1}{2}\frac{\mathrm{d}}{\mathrm{d}t}\norm{\nabla u}^2 + \mu\norm{\Delta u}^2 &+ \beta\tcalka{ \abs{u}^{r - 1}\abs{\nabla u}^2 } \leq \tcalka{ \abs{u}\abs{\nabla u}\abs{-\Delta u} } \\ \nonumber
&\leq \frac{1}{2}\lr{\frac{1}{\mu}\tcalka{ \abs{u}^2\abs{\nabla u}^2 } + \mu\tcalka{ \abs{\Delta u}^2 }}
\end{align}
and hence
\begin{align} \label{globalr}
\frac{\mathrm{d}}{\mathrm{d}t}\norm{\nabla u}^2 &+ \mu\norm{\Delta u}^2 + 2\beta\tcalka{ \abs{u}^{r - 1}\abs{\nabla u}^2 } \leq \frac{1}{\mu}\tcalka{ \abs{u}^2\abs{\nabla u}^2 }.
\end{align}

Now we observe the following estimate for $r > 3$:
\begin{align} \nonumber
\tcalka{ \abs{u}^2\abs{\nabla u}^2 } &= \tcalka{ \lr{\abs{u}^2\abs{\nabla u}^{4/(r - 1)}}\lr{\abs{\nabla u}^{2(r - 3)/(r - 1)}} } \\ \nonumber
&\leq \lr{\tcalka{ \abs{u}^{r - 1}\abs{\nabla u}^{2}} }^{2/(r - 1)}\lr{\tcalka{ \abs{\nabla u}^{2}} }^{(r - 3)/(r - 1)} \\ \label{rgeqtrzy}
&\leq \beta\mu\lr{\tcalka{ \abs{u}^{r - 1}\abs{\nabla u}^2 }} + c(\beta, \mu, r)\lr{\tcalka{ \abs{\nabla u}^2 }} .
\end{align}
In the above we used H\"{o}lder's and Young's inequalities with the same exponents $(r - 1)/2$ and $(r - 1)/(r - 3)$. The value of the constant $c(\beta, \mu, r)$ can be computed explicitly
$$ c(\beta, \mu, r) = \lr{\frac{2}{\beta\mu(r - 1)}}^{2/(r - 3)}\lr{\frac{r - 3}{r - 1}}. $$

Plugging the estimate $\lr{\ref{rgeqtrzy}}$ into $\lr{\ref{globalr}}$ gives
\begin{align} \label{extraregularity}
\frac{\mathrm{d}}{\mathrm{d}t}\norm{\nabla u}^2 + \mu\norm{\Delta u}^2 + \beta\tcalka{ \abs{u}^{r - 1}\abs{\nabla u}^2 } \leq \frac{c({\beta, \mu, r})}{\mu}\norm{\nabla u}^2.
\end{align}

In particular, we have
\begin{align} \nonumber
\frac{\mathrm{d}}{\mathrm{d}t}\norm{\nabla u}^2 \leq \frac{c({\beta, \mu, r})}{\mu}\norm{\nabla u}^2.
\end{align}
Application of Gronwall's Lemma yields that $\norm{\nabla u}^2$ stays bounded on arbitrarily large time interval $\lra{0, T}$. Additionally, since $\tcalka{u} = 0$, we have the following relations
$$ \norm{u}_{V} \leq c\norm{\nabla u} \quad \mbox{and} \quad \norm{u}_{V^2} \leq c\norm{\Delta u}, $$
so in particular, $u \in L^{\infty}(0, T; V)$. Then one infers from $(\ref{extraregularity})$ that $\int_{0}^{T}{\norm{\Delta u}^2} < \infty$. 
Therefore, $u$ is indeed a strong solution on the time interval $[0, T]$ for all $T > 0$.

Additional regularity $(\ref{spaces})$ for the function $u$ follows now from the inequality $(\ref{extraregularity})$, Lemma $\ref{witeklemmadwa}$, and the estimate $(\ref{maleklemma})$.
\end{proof}

\section{Global existence for $r = 3$ and for large coefficients $4\mu\beta \geq 1$}

From now on we consider only the critical case of the convective Brinkman--Forchheimer equations $(r = 3)$. We prove here global-in-time existence of strong solutions of $(\ref{alphazerocrit})$ for all initial conditions $u_0 \in V$, when the product $\mu\beta$ is sufficiently large. From the physical point of view this is not a surprising result. It means that when both the viscosity of a fluid and the porosity of a porous medium are large enough, then the corresponding flow stays bounded and regular. What is more interesting is the fact that when the viscosity is small one can still obtain a regular solution by letting the porosity to be large, and vice versa.
\begin{theorem} \label{globalcoef}
For every initial condition $u_0 \in V$, there exists a global-in-time strong solution of the critical CBF equations $(\ref{alphazerocrit})$, provided that $4\mu\beta \geq 1$, i.e.\ both the viscosity and porosity are not too small.
\end{theorem}

Again, we present here only a priori estimates which can be made rigorous by the means of a standard Galerkin approximation argument, see \cite{ConstantinFoias}, \cite{Galdi}, or \cite{Temam}, for example.

\begin{proof}
Multiplying the equation $(\ref{alphazerocrit})$ by $-\Delta u$ and integrating over $\mathbb{T}^3$, we obtain
\begin{align} \nonumber
\frac{1}{2}\frac{\mathrm{d}}{\mathrm{d}t}\norm{\nabla u}^2 + \mu\norm{\Delta u}^2 + \beta\dual{\abs{u}^{2}u}{-\Delta u} \leq \abs{\dual{(u \cdot \nabla)u}{-\Delta u}}.
\end{align}
Applying Lemma $\ref{witeklemma}$, we have
\begin{align} \label{coeffest}
\frac{1}{2}\frac{\mathrm{d}}{\mathrm{d}t}\norm{\nabla u}^2 + \mu\norm{\Delta u}^2 + \beta\tcalka{\abs{\nabla u}^2\abs{u}^{2}} \leq \tcalka{\abs{u}\abs{\nabla u}\abs{-\Delta u}}.
\end{align}
We want to estimate the right-hand side in such a way to absorb it with the terms on the left-hand side. Using the Cauchy--Schwarz and Young inequalities we obtain
\begin{align} \nonumber
\tcalka{\abs{u}\abs{\nabla u}\abs{\Delta u}} &\leq \lr{\tcalka{\abs{\nabla u}^2\abs{u}^{2}}}^{1/2}\lr{\tcalka{\abs{\Delta u}^2}}^{1/2} \\ \nonumber
&\leq \frac{\theta}{2}\tcalka{\abs{\nabla u}^2\abs{u}^{2}} + \frac{1}{2\theta}\tcalka{\abs{\Delta u}^2},
\end{align}
for some positive number $\theta > 0$. We use this estimate in the inequality $(\ref{coeffest})$ and~then move all the terms to the left-hand side to obtain
\begin{align} \nonumber
\frac{1}{2}\frac{\mathrm{d}}{\mathrm{d}t}\norm{\nabla u}^2 + \lr{\mu - \frac{1}{2\theta}}\norm{\Delta u}^2 + \lr{\beta - \frac{\theta}{2}}\tcalka{\abs{\nabla u}^2\abs{u}^{2}} \leq 0.
\end{align}
From the above we see that the norm $\norm{\nabla u(t)}^2$ is not increasing in time, provided that
$$ \mu - \frac{1}{2\theta} \geq 0 \quad \mbox{and} \quad \beta - \frac{\theta}{2} \geq 0  \quad \Longleftrightarrow \quad \mu\beta \geq \frac{1}{4}. $$
Hence, there is no blow-up and the strong solution originating from the initial condition $u_0 \in V$ exists for all times $t > 0$.
\end{proof}

We note that the above argument works only for the critical exponent $r = 3$. For other values of $r \in [1, 3)$ we are not able to balance the exponents in the correct way to absorb the convective term on the left-hand side of $(\ref{coeffest})$.

\section{Energy equality for critical case $r = 3$} \label{s:mainresult}

In this section we will prove Theorem $\ref{twr:ee}$, mostly following the proof of Theorem $4.1$ in \cite{Galdi}. The main idea is to use a weak solution as a test function. We cannot do this directly since $u$ is not sufficiently regular in space or time. Therefore, we regularise in time the finite-dimensional approximations of a weak solution and pass to the~limit with both the regularisation and spatial approximation parameters. To this end we recall here some standard facts of the theory of mollification.

Let $\eta(t)$ be an even, positive, smooth function with compact support contained in the interval $(-1, 1)$, such that
$$ \int_{-\infty}^{\infty}{\eta(s) \, \mathrm{d}s} = 1. $$
We denote by $\eta_{h}$ a family of mollifiers connected with the function $\eta$, i.e. 
$$ \eta_{h}(s) := h^{-1}\eta(s/h) \quad \mbox{for} \quad h > 0. $$
In particular, we have
\begin{equation} \label{half}
\int_{0}^{h}{\eta_{h}(s) \, \mathrm{d}s} = \frac{1}{2}.
\end{equation}
For any function $v \in L^q(0, T; X)$, where $X$ is a Banach space, $q \in [1, \infty)$, we denote its mollification in time by $v^{h}$
$$ v^{h}(s) := (v  \ast \eta_{h})(s) = \int_{0}^{T}{v(\tau)\eta_{h}(s - \tau) \, \mathrm{d}\tau} \quad{for} \quad h \in (0, T). $$

We have the following properties of this mollification (see Lemma $2.5$ in \cite{Galdi}).
\begin{lemma} \label{timemollifier}
Let $w \in L^q({0, T; X})$, $1 \leq q < \infty$, for some Banach space $X$. Then $w^{h} \in C^k({[0, T); X})$ for all $k \geq 0$. Moreover,
$$ \lim_{h \to 0}{\norm{w^{h} - w}_{L^q(0, T; X)}} = 0. $$
Finally, if $\lrb{w_n}_{n = 1}^{\infty}$ converges to $w$ in $L^q({0, T; X})$, then
$$ \lim_{n \to \infty}{\norm{w^{h}_n - w^{h}}_{L^q({0, T; X})}} = 0. $$
\end{lemma}

Since our domain is the three-dimensional torus, we can approximate functions in $L^p$ spaces using carefully truncated Fourier expansions. We state this more precisely in the~following theorem (see e.g.\ Theorem $1.6$ in \cite{RRSbook} for more details).
\begin{theorem} \label{fourierinlp}
Let $\mathcal{Q}_n := [-n, n]^3 \cap \mathbb{Z}^3$. For every $w \in L^1(\mathbb{T}^3)$ and every $n \in \mathbb{N}$ define
\begin{equation} \label{fouriertruncation}
S_n(w) := \sum_{k \in \mathcal{Q}_n}\hat{w}_ke^{ik\cdot x}, 
\end{equation}
where the Fourier coefficients $\hat{w}_k$ are given by
$$ \hat{w}_k := \frac{1}{\abs{\mathbb{T}^3}}\int_{\mathbb{T}^3}{w(x)e^{-ik\cdot x} \, \mathrm{d}x}. $$
Then for every $1 < p < \infty$ there is a constant $c_p$, independent of $n$, such that
$$ \norm{S_n(w)}_{L^p(\mathbb{T}^3)} \leq c_p\norm{w}_{L^p(\mathbb{T}^3)} \quad \mbox{for all} \quad w \in L^p(\mathbb{T}^3) $$
and
$$ \norm{S_n(w) - w}_{L^p(\mathbb{T}^3)} \to 0 \quad \mbox{as} \quad n \to \infty. $$
\end{theorem}

We remark that the convergence $\tilde{S}_n(w) \to w$ in $L^p(\mathbb{T}^3)$ does not hold in general if we sum over Fourier modes with $\abs{k} \leq n$, 
$$ \tilde{S}_n(w) := \sum_{\abs{k} \leq n}\hat{w}_ke^{ik\cdot x} $$
(see e.g. Section $1.6$ in \cite{RRSbook} for a brief discussion of this result and for some further references).

Now we can prove the following density result which will be used in the~proof of Theorem $\ref{twr:ee}$.
\begin{lemma} \label{density}
$\mathcal{D}_{\sigma}({\mathbb{T}^3 \times [0, T)})$ is dense in $L^{4}({0, T; X_4(\mathbb{T}^3)}) \cap L^{2}({0, T; V})$.
\end{lemma}
\begin{proof}
Let $w \in L^{4}({0, T; X_4(\mathbb{T}^3)}) \cap L^{2}({0, T; V})$ and define
$$ w_{n}^{h}(x, t) := S_n(w(x, t)^h) = \sum_{k \in \mathcal{Q}_n}{\hat{w}_k^{h}(t)}e^{ik\cdot x} \quad \mbox{for} \quad h \in(0, T), $$
where $S_n$ is the same as in $(\ref{fouriertruncation})$.
Clearly, $w_{n}^{h} \in \mathcal{D}_{\sigma}({\mathbb{T}^3 \times [0, T)})$.
By Theorem $\ref{fourierinlp}$ we have
\begin{equation} \label{convgalerkinl4}
\lim_{n \to \infty}{\norm{w_{n}^{h}(t) - w^{h}(t)}_{L^4(\mathbb{T}^3)}^4} = 0
\end{equation}
and
\begin{equation} \label{convgalerkinv}
\lim_{n \to \infty}{\norm{w_{n}^{h}(t) - w^{h}(t)}_V^2} = \lim_{n \to \infty}\lr{{\norm{w_{n}^{h}(t) - w^{h}(t)}^2} + {\norm{\nabla w_{n}^{h}(t) - \nabla w^{h}(t)}^2}} = 0
\end{equation}
for all $t \in [0, T)$.
 By Lemma $\ref{timemollifier}$, for a given $\varepsilon > 0$, we can choose $h > 0$ so small that
\begin{align} \label{withepsilon}
\calkat{\norm{w^{h}(t) - w(t)}_{L^4(\mathbb{T}^3)}^{4}} < \varepsilon \quad \mbox{and} \quad \calkat{\norm{w^{h}(t) - w(t)}_{V}^{2}} < \varepsilon.
\end{align}
On the other hand, from $(\ref{convgalerkinl4})$, $(\ref{convgalerkinv})$, and the Lebesgue Dominated Convergence Theorem, we have that for all fixed $h \in (0, T)$
\begin{align} \label{withn}
\lim_{n \to \infty}\calkat{\norm{w_{n}^{h}(t) - w^{h}(t)}_{L^4(\mathbb{T}^3)}^{4}} = 0 \quad \mbox{and} \quad \lim_{n \to \infty}\calkat{\norm{w_{n}^{h}(t) - w^{h}(t)}_{V}^{2}} = 0,
\end{align}
since, respectively, $\norm{w_{n}^{h}(t)}_{L^4(\mathbb{T}^3)} \leq c\norm{w^{h}(t)}_{L^4(\mathbb{T}^3)}$, and $\norm{w_{n}^{h}(t)}_V \leq c\norm{w^{h}(t)}_{V}$ for all $n \in \mathbb{N}$ and $t \in [0, T)$, and
$$ w^{h} \in L^4(0, T; X_4(\mathbb{T}^3)) \cap L^2(0, T; V). $$
Thus, the lemma follows from the relations $(\ref{withepsilon})$, $(\ref{withn})$ and the triangle inequality.
\end{proof}

Now, we are in a position to prove Theorem $\ref{twr:ee}$.
\begin{proof}
Let $\lrb{u_n}_{n = 1}^{\infty} \subset \mathcal{D}_{\sigma}({\mathbb{T}^3 \times [0, T)})$ be a sequence converging to a weak solution $u$ in $L^4({0, T; X_4(\mathbb{T}^3)})$ and in $L^2({0, T; V})$, see Lemma $\ref{density}$. For every fixed time instant $t_1\in (0, T)$, we choose in Definition \ref{defi:slabecbfr} (with $t_0 = 0$) a sequence of test functions
\begin{align} \nonumber
\varphi_{n}^{h}(x ,s) &:= \lr{u_n(x, s)\chi_{[0, t_1]}(s)}^{h} = \lr{u_n\chi_{[0, t_1]} \ast \eta_{h}}(x, s) \\ \nonumber
&= \int_{0}^{T}{u_n(x, \tau)\chi_{[0, t_1]}(\tau)\eta_{h}(s - \tau) \, \mathrm{d}\tau} = \int_{0}^{t_1}{u_n(x, \tau)\eta_{h}(s - \tau) \, \mathrm{d}\tau}
\end{align}
for $(x, s) \in \mathbb{T}^3 \times [0, T)$, with the parameter $h$ satisfying the following conditions:
$$ 0 < h < T - t_1 \quad \mbox{and} \quad h < t_1. $$
We obtain a sequence of equations
\begin{align} \nonumber
-\int_{0}^{t_1}{\dual{u(s)}{\partial_t(u_n\chi_{[0, t_1]})^{h}(s)}\, \mathrm{d}s} &+ \mu\int_{0}^{t_1}{\dual{\nabla u(s)}{\nabla (u_n\chi_{[0, t_1]})^{h}(s)}\, \mathrm{d}s} \\ \nonumber
+ \int_{0}^{t_1}{\dual{(u(s) \cdot \nabla)u(s)}{(u_n\chi_{[0, t_1]})^{h}(s)}\, \mathrm{d}s} &+ \beta\int_{0}^{t_1}{\dual{\abs{u(s)}^{2}u(s)}{(u_n\chi_{[0, t_1]})^{h}(s)}\, \mathrm{d}s} \\ \label{seqofweakforms}
= -\dual{u(t_1)}{(u_n\chi_{[0, t_1]})^{h}(t_1)} &+ \dual{u(0)}{(u_n\chi_{[0, t_1]})^{h}(0)}.
\end{align}

Note that our choice of $h$ ensures that $\varphi_{n}^{h}(x ,T) = 0$. Additionally, observe that the functions $\varphi_{n}^{h}$ are divergence-free, since $\dv{\varphi_{n}^{h}} =(\dv{\varphi_n})^{h} = 0$, so indeed $\varphi_n^{h} \in \mathcal{D}_{\sigma}({\mathbb{T}^3 \times [0, T)})$.

We want to pass to the limit in $(\ref{seqofweakforms})$ as $n \to \infty$. To this end, using H\"older's inequality and Lemma $\ref{timemollifier}$, we observe the following estimates for the nonlinear terms:
\begin{align} \nonumber
&\abs{\int_{0}^{t_1}{\dual{(u(s) \cdot \nabla)u(s)}{(u_n\chi_{[0, t_1]})^{h}(s)} \, \mathrm{d}s} - \int_{0}^{t_1}{\dual{(u(s) \cdot \nabla)u(s)}{(u\chi_{[0, t_1]})^{h}(s)} \, \mathrm{d}s}} \\ \nonumber
&\leq \int_{0}^{t_1}{{\norm{u(s)}_{L^4(\mathbb{T}^3)}\norm{\nabla u(s)}\norm{(u_n\chi_{[0, t_1]})^{h}(s) - (u\chi_{[0, t_1]})^{h}(s)}_{L^4(\mathbb{T}^3)}}} \, \mathrm{d}s \\ \label{convectiveterm}
&\leq \norm{u}_{L^4({0, T; X_4(\mathbb{T}^3)})}\norm{u}_{L^2({0, T; V})}\norm{(u_n\chi_{[0, t_1]})^{h} - (u\chi_{[0, t_1]})^{h}}_{L^4({0, T; X_4(\mathbb{T}^3)})} \to 0
\end{align}
as $n \to \infty$, and
\begin{align} \nonumber
&\abs{\int_{0}^{t_1}{\dual{\abs{u(s)}^2u(s)}{(u_n\chi_{[0, t_1]})^{h}(s)} \, \mathrm{d}s} - \int_{0}^{t_1}{\dual{\abs{u(s)}^2u(s)}{(u\chi_{[0, t_1]})^{h}(s)} \, \mathrm{d}s}} \\ \nonumber
&\leq \int_{0}^{t_1}{{\norm{u(s)}^3_{L^4(\mathbb{T}^3)}\norm{(u_n\chi_{[0, t_1]})^{h}(s) - (u\chi_{[0, t_1]})^{h}(s)}_{L^4(\mathbb{T}^3)}}} \, \mathrm{d}s \\ \label{absorptionterm}
&\leq \norm{u}^3_{L^4({0, T; X_4(\mathbb{T}^3)})}\norm{(u_n\chi_{[0, t_1]})^{h} - (u\chi_{[0, t_1]})^{h}}_{L^4({0, T; X_4(\mathbb{T}^3)})} \to 0 \quad \mbox{as} \quad n \to \infty.
\end{align}
Estimating the linear terms in a standard way and using $({\ref{convectiveterm}})$, $({\ref{absorptionterm}})$ we can pass in the weak formulation $(\ref{seqofweakforms})$ to the limit as $n \to \infty$. We arrive at the identity
\begin{align} \nonumber
-\int_{0}^{t_1}{\dual{u(s)}{\partial_t(u\chi_{[0, t_1]})^{h}(s)}\, \mathrm{d}s} &+ \mu\int_{0}^{t_1}{\dual{\nabla u(s)}{\nabla (u\chi_{[0, t_1]})^{h}(s)}\, \mathrm{d}s} \\ \nonumber
+ \int_{0}^{t_1}{\dual{(u(s) \cdot \nabla)u(s)}{(u\chi_{[0, t_1]})^{h}(s)}\, \mathrm{d}s} &+ \beta\int_{0}^{t_1}{\dual{\abs{u(s)}^{2}u(s)}{(u\chi_{[0, t_1]})^{h}(s)}\, \mathrm{d}s} \\ \nonumber
= -\dual{u(t_1)}{(u\chi_{[0, t_1]})^{h}(t_1)} &+ \dual{u(0)}{(u\chi_{[0, t_1]})^{h}(0)}.
\end{align}
Since the function $\eta_{h}$ is even in $(-h, h)$, we have $\dot{\eta}_{h}(r) = -\dot{\eta}_{h}(-r)$ and so
\begin{align} \nonumber
\int_{0}^{t_1}{\dual{u(s)}{\partial_t(u\chi_{[0, t_1]})^{h}(s)}\, \mathrm{d}s} &= \int_{0}^{t_1}{\lr{\int_{0}^{t_1}{\dot{\eta}_{h}(s - \tau)\dual{u(s)}{u(\tau)} \, \mathrm{d}\tau}} \, \mathrm{d}s} \\ \nonumber
&= -\int_{0}^{t_1}{\lr{\int_{0}^{t_1}{\dot{\eta}_{h}(\tau - s)\dual{u(s)}{u(\tau)} \, \mathrm{d}\tau}} \, \mathrm{d}s} \\ \nonumber
&= -\int_{0}^{t_1}{\lr{\int_{0}^{t_1}{\dot{\eta}_{h}(\tau - s)\dual{u(\tau)}{u(s)} \, \mathrm{d}\tau}} \, \mathrm{d}s} \\ \nonumber
&= -\int_{0}^{t_1}{\lr{\int_{0}^{t_1}{\dot{\eta}_{h}(\tau - s)\dual{u(\tau)}{u(s)} \, \mathrm{d}s}} \, \mathrm{d}\tau} = 0.
\end{align}

Next, by repeating the arguments in $({\ref{convectiveterm}})$, $({\ref{absorptionterm}})$ with $(u\chi_{[0, t_1]})^{h}$ in place of $(u_n\chi_{[0, t_1]})^{h}$ and $u\chi_{[0, t_1]}$ in place of $(u\chi_{[0, t_1]})^{h}$, we obtain
\begin{align} \nonumber
\lim_{h \to 0}\int_{0}^{t_1}{\dual{(u(s) \cdot \nabla)u(s)}{(u\chi_{[0, t_1]})^{h}(s)} \, \mathrm{d}s} &= \int_{0}^{t_1}{\dual{(u(s) \cdot \nabla)u(s)}{(u\chi_{[0, t_1]})(s)} \, \mathrm{d}s} \\ \nonumber
&= \int_{0}^{t_1}{\dual{(u(s) \cdot \nabla)u(s)}{u(s)} \, \mathrm{d}s}= 0, \\ \nonumber
\lim_{h \to 0}\int_{0}^{t_1}{\dual{\abs{u(s)}^2u(s)}{(u\chi_{[0, t_1]})^{h}(s)} \, \mathrm{d}s} &= \int_{0}^{t_1}{\dual{\abs{u(s)}^2u(s)}{(u\chi_{[0, t_1]})(s)} \, \mathrm{d}s}, \\ \nonumber
\lim_{h \to 0}\int_{0}^{t_1}{\dual{\nabla u(s)}{\nabla (u\chi_{[0, t_1]})^{h}(s)} \, \mathrm{d}s} &= \int_{0}^{t_1}{\dual{\nabla u(s)}{\nabla(u\chi_{[0, t_1]})(s)} \, \mathrm{d}s},
\end{align}
which give us
\begin{align} \nonumber
\mu\int_{0}^{t_1}{\norm{\nabla u(s)}^2 \, \mathrm{d}s} + \beta\int_{0}^{t_1}{\norm{u(s)}_{L^4(\mathbb{T}^3)}^4 \, \mathrm{d}s} &= -\lim_{h \to 0}\dual{u(t_1)}{(u\chi_{[0, t_1]})^{h}(t_1)} \\ \nonumber
&+ \lim_{h \to 0}\dual{u(0)}{(u\chi_{[0, t_1]})^{h}(0)}.
\end{align}
Finally, from the fact that $u$ is $L^2$-weakly continuous in time and from $(\ref{half})$, we have
\begin{align} \nonumber
&\dual{u(t_1)}{(u\chi_{[0, t_1]})^{h}(t_1)} = \int_{0}^{T}{\eta_{h}(s)\chi_{[0, t_1]}(t_1 - s)\dual{u(t_1)}{u(t_1- s)} \, \mathrm{d}s} \\ \nonumber
&= \int_{0}^{t_1}{\eta_{h}(s)\dual{u(t_1)}{u(t_1 - s)} \, \mathrm{d}s} = \int_{0}^{h}{\eta_{h}(s)\dual{u(t_1)}{u(t_1 - s)} \, \mathrm{d}s} \\ \nonumber
&= \frac{1}{2}\norm{u(t_1)}^2 + \int_{0}^{h}{\eta_{h}(s)\dual{u(t_1)}{u(t_1 - s) - u(t_1)} \, \mathrm{d}s} \to \frac{1}{2}\norm{u(t_1)}^2
\end{align}
as $h \to 0$. In the same manner we show that
\begin{align} \nonumber
\dual{u(0)}{(u\chi_{[0, t_1]})^{h}(0)} \to \frac{1}{2}\norm{u(0)}^2 \quad \mbox{as} \quad h \to 0.
\end{align}
Finally, we obtain the identity
\begin{align} \label{proofee}
\frac{1}{2}\norm{u(t_1)}^2 + \mu\int_{0}^{t_1}{\norm{\nabla u(s)}^2 \, \mathrm{d}s} + \beta\int_{0}^{t_1}{\norm{u(s)}_{L^4(\mathbb{T}^3)}^4 \, \mathrm{d}s} = \frac{1}{2}\norm{u(0)}^2
\end{align}
for all $t_1 \in (0, T).$
The energy equality $(\ref{equality})$ follows easily from $(\ref{proofee})$.

Now we will prove the last part of the theorem, namely that all weak solutions of the~critical CBF equations $(\ref{alphazerocrit})$ are continuous into $L^2$ with respect to time, i.e.
\begin{equation} \label{strongconv}
\norm{u(t) - u(t_0)} \to 0 \quad \mbox{as} \quad t \to t_0
\end{equation}
for all $t_0 \in [0, T)$.

First, we remark that all weak solutions of $(\ref{cbfequation})$ are $L^2$-weakly continuous with respect to time
\begin{equation} \label{weakctn}
u(t) \rightharpoonup u(t_ 0) \quad \mbox{as} \quad t \to t_0
\end{equation}
for all $t_0 \in [0, T)$. This can be proved similarly as for the Navier--Stokes equations and follows immediately from the definition of a weak solution (see e.g.\ Lemma $2.2$ in \cite{Galdi} for the details).
	
Now, let $u$ be a weak solution of $(\ref{alphazerocrit})$. It suffices to have the energy inequality to obtain convergence of the norms, since from $(\ref{energyineq})$ we deduce immediately that
$$ \limsup_{t \to t_0} \norm{u(t)}^2 \leq \norm{u(t_0)}^2 $$
and from the weak continuity $(\ref{weakctn})$ we have
$$ \liminf_{t \to t_0} \norm{u(t)}^2 \geq \norm{u(t_0)}^2. $$
Therefore, it follows from the first part of Theorem \ref{twr:ee} that for all weak solutions of $(\ref{alphazerocrit})$ we have
\begin{equation} \label{cbfweaknormconv}
u(t) \rightharpoonup u(t_0) \quad \mbox{and} \quad \norm{u(t)} \to \norm{u(t_0)} \quad \mbox{as} \quad t \to t_0.
\end{equation}
Since in a Hilbert space weak convergence and convergence of norms implies  strong convergence, the result $(\ref{strongconv})$ follows immediately from $(\ref{cbfweaknormconv})$.
\end{proof}

\section{Strong global attractor}

We have now proved that all weak solutions of the convective Brinkman--\linebreak Forchheimer equations with critical exponent $r = 3$ satisfy the energy equality. As a consequence, we obtained unconditional continuity of all weak solutions into~$L^2$.

In the paper of Ball \cite{Ball} it was shown for the three-dimensional incompressible Navier--Stokes equations that strong $L^2$-continuity leads to the existence of a global attractor in the phase space $H$. Due to technical difficulties we cannot apply his method of generalised semiflows to our problem. Instead, we make use of the theory of evolutionary systems due to Cheskidov \cite{Cheskidov} to show existence of a strong global attractor for the critical convective Brinkman--Forchheimer equations $(\ref{alphazerocrit})$. First, we introduce some necessary notation from \cite{Cheskidov}.

\subsection{Evolutionary systems}

Let $(X, d_s(\cdot, \cdot))$ be a metric space endowed with a~metric $d_s$, which will be referred to as a strong metric. Let $d_w(\cdot, \cdot)$ be another metric on $X$ satisfying the following conditions:
\begin{enumerate}[label = \arabic*.]
\item \label{wcompact} $X$ is $d_w$-compact.

\item \label{simplyw} If $d_s(u_n, v_n) \to 0$ as $n \to \infty$ for some $u_n, v_n \in X$, then $d_w(u_n, v_n) \to 0$ as $n \to \infty$.
\end{enumerate}
Note that any $d_s$-compact set is $d_w$-compact and any weakly closed set is strongly closed.

Let $C([a, b]; X_{\bullet})$, where $\bullet \in \lrb{s, w}$, be the space of $d_{\bullet}$-continuous $X$-valued functions on $[a, b]$ endowed with the metric
$$ d_{C([a, b]; X_{\bullet})}(u, v) := \sup_{t \in [a, b]}\lrb{d_{\bullet}(u(t), v(t))}. $$
Let also $C([a, \infty); X_{\bullet})$ be the space of $d_{\bullet}$-continuous $X$-valued functions on $[a, \infty)$ endowed with the metric
$$ d_{C([a, \infty); X_{\bullet})}(u, v) := \sum_{T \in \mathbb{N}}{\frac{1}{2^T}\frac{\sup\lrb{d_{\bullet}(u(t), v(t)) : a \leq t \leq a + T}}{1 + \sup\lrb{d_{\bullet}(u(t), v(t)) : a \leq t \leq a + T}}}. $$

To define an evolutionary system, first let
$$ \mathcal{T} := \lrb{I: I = [T, \infty) \subset \mathbb{R} \mbox{ for } T \in \mathbb{R}, \mbox{ or } I = (-\infty, \infty)}, $$
and for each $I \subset \mathcal{T}$, let $\mathcal{F}(I)$ denote the set of all $X$-valued functions on $I$.
\begin{definition}
A map $\mathcal{E}$ that associates to each $I \in \mathcal{T}$ a subset $\mathcal{E}(I) \subset \mathcal{F}(I)$ will be called an \emph{evolutionary system} if the following conditions are satisfied:
\begin{enumerate}
\item $\mathcal{E}([0, \infty)) \neq \emptyset$.
	
\item $\mathcal{E}(I + s) = \lrb{u(\cdot) : u(\cdot - s) \in \mathcal{E}(I)}$ for all $s \in \mathbb{R}$.
	
\item $\lrb{u(\cdot)|_{I_2} : u(\cdot) \in \mathcal{E}(I_1)} \subset \mathcal{E}(I_2)$ for all pairs $I_1, I_2 \in \mathcal{T}$, such that $I_2 \subset I_1$.
	
\item $\mathcal{E}((-\infty, \infty)) = \lrb{u(\cdot) : u(\cdot)|_{[T, \infty)} \in \mathcal{E}([T, \infty)) \mbox{ for all } T \in \mathbb{R}}$.
\end{enumerate}
\end{definition}
We will refer to $\mathcal{E}(I)$ as \textit{the set of all trajectories} on the time interval $I$. Trajectories in $\mathcal{E}((-\infty, \infty))$ will be called \emph{complete}. To relate the notion of evolutionary systems with the classical notion of semiflows, let $P(X)$ be the set of all subsets of $X$. For every $t \geq 0$, define a map $R(t) : P(X) \to P(X)$, such that
$$ R(t)A := \lrb{u(t) : u \in A, u \in \mathcal{E}([0, \infty))} \quad \mbox{for} \quad A \subset X. $$
Note that the assumptions on $\mathcal{E}$ imply that $R(t)$ enjoys the following property:
$$ R(t + s)A \subset R(t)R(s)A, \quad A \subset X, \ t,s \geq 0. $$
One can check that a semiflow defines an evolutionary system (see details in \cite{Cheskidov}).

Furthermore, we will consider evolutionary systems $\mathcal{E}$ satisfying the following assumptions:

\begin{enumerate}[label=(\subscript{A}{\arabic*})]
\item \label{Aone} (\emph{Weak compactness}) $\mathcal{E}([0, \infty))$ is a compact set in $C([0, \infty); X_w)$.
	
\item \label{Atwo} (\emph{Energy inequality}) Assume that $X$ is a bounded set in some uniformly convex Banach space $H$ with the norm denoted by $\norm{\cdot}$, such that 
$$ d_s(x, y) = \norm{x - y} \quad \mbox{for} \quad x, y \in X. $$
Assume also that for any $\varepsilon > 0$, there exists $\delta > 0$, such that for every $u \in \mathcal{E}([0, \infty))$ and $t > 0$,
$$ \norm{u(t)} \leq \norm{u(t_0)} + \varepsilon, $$
for $t_0 \in (t - \delta, t)$.
	
\item \label{Athree} (\emph{Strong convergence a.e.}) Let $u, u_n \in \mathcal{E}([0, \infty))$ be such that $u_n \to u$ in $C([0, T]; X_w)$ for some $T > 0$. Then $u_n(t) \to u(t)$ strongly for a.e. $t \in [0, T]$.
\end{enumerate}

Consider an arbitrary evolutionary system $\mathcal{E}$. For a set $A \subset X$ and $r > 0$, denote an open ball
$$ B_{\bullet}(A, r) := \lrb{u \in X : d_{\bullet}(u, A) < r}, $$
where
$$ d_{\bullet}(u, A) := \inf_{x \in A}\lrb{d_{\bullet}(u, x)}. $$

\begin{definition}
A set $A \subset X$ \textit{uniformly attracts} a set $B \subset X$ in the $d_{\bullet}$-metric if for any $\varepsilon > 0$ there exists $t_0$, such that
$$ R(t)B \subset B_{\bullet}(A, \varepsilon), \quad \forall \, t \geq t_0. $$
\end{definition}

\begin{definition}
A set $A \subset X$ is a \textit{$d_{\bullet}$-attracting set} if it uniformly attracts $X$ in the $d_{\bullet}$-metric.
\end{definition}

\begin{definition}
A set $\mathcal{A}_{\bullet} \subset X$ is a \textit{$d_{\bullet}$-global attractor} if $\mathcal{A}_{\bullet}$ is a minimal $d_{\bullet}$-closed, $d_{\bullet}$-attracting set. 
\end{definition}

\begin{theorem} \label{weakglobalattractor}
Every evolutionary system possesses a weak global attractor $\mathcal{A}_{w}$. Moreover, if a strong global attractor $\mathcal{A}_{s}$ exists, then $\overline{\mathcal{A}_{s}}^{w} = \mathcal{A}_{w}$.
\end{theorem}

\begin{definition}
The evolutionary system $\mathcal{E}$ is \textit{asymptotically compact} if for any sequence $t_n \to \infty$ as $n \to \infty$, and any $x_n \in R(t_n)X$, the sequence $\lrb{x_n}_{n = 1}^{\infty}$ is relatively strongly compact.
\end{definition}

\begin{theorem} \label{asymptoticcompactness}
If an evolutionary system $\mathcal{E}$ is asymptotically compact, then $\mathcal{A}_{w}$ is a strongly compact strong global attractor $\mathcal{A}_s$.
\end{theorem}

\begin{theorem} \label{globalattractor}
Let $\mathcal{E}$ be an evolutionary system satisfying \emph{\ref{Aone}}, \emph{\ref{Atwo}}, and \emph{\ref{Athree}}. If every complete trajectory is strongly continuous, i.e.\ if
\begin{equation*}
\mathcal{E}((-\infty, \infty)) \subset C((-\infty, \infty); X_s),
\end{equation*}
then $\mathcal{E}$ is asymptotically compact.
\end{theorem}

\subsection{Application to the critical CBF equations}

In \cite{Cheskidov} it was shown that all Leray--Hopf weak solutions of the space-periodic 3D NSE form an evolutionary system $\mathcal{E}$ satisfying \ref{Aone}, \ref{Atwo}, and \ref{Athree}.
We will show here that this result can be extended also to all weak solutions of the critical CBF equations. Actually, this is true also for all Leray--Hopf weak solutions of the CBF equations with the~exponent $r \geq 1$. However, only for the critical case $r = 3$, due to Theorem $\ref{twr:ee}$, there is no need to distinguish between these solutions. We begin by setting our problem into the framework of evolutionary systems.

We define the strong and weak distances by
$$ d_s(u, v) := \norm{u - v}, \quad d_w(u, v) := \sum_{k \in {\mathbb{Z}}^3}{\frac{1}{2^k}\frac{\abs{\hat{u}_k - \hat{v}_k}}{1 + \abs{\hat{u}_k - \hat{v}_k}}}, \quad u,v \in H, $$
where $\hat{u}_k$ and $\hat{v}_k$ are the Fourier coefficients of $u$ and $v$ respectively.

\begin{definition}
A ball $B_{\bullet}(0, r) \subset H$ is called \textit{a $d_{\bullet}$-absorbing ball} if for any bounded set $A \subset H$, there exists $t_0$, such that
$$ R(t)A \subset B_{\bullet}(0, r)\quad \forall \, t \geq t_0. $$
\end{definition}

For the 3D NSE it is well known that there exists an absorbing ball (see e.g.\ \cite{ConstantinFoias}). The same can be easily proved for the CBF equations.
\begin{proposition}
There exists a radius $R > 0$ such that the ball $B_s(0, R) \subset H$ is a~$d_s$-absorbing ball for the CBF equations $(\ref{alphazero})$.
\end{proposition}

Let $X$ be a closed absorbing ball for the critical CBF equations $(\ref{alphazerocrit})$,
$$ X := \lrb{u \in H : \norm{u} \leq R}, $$
which is also weakly compact. Then for any bounded set $A \in H$ there exists a time $t_0$, such that
$$ u(t) \in X \quad \mbox{for all} \quad t \geq t_0, $$
for every weak solution $u(t)$ with the initial condition $u(0) \in A$.

Contrary to the NSE, all weak solutions of the critical CBF equations satisfy the energy inequality. In fact they verify the stronger condition $(\ref{equality})$. Therefore, we consider an evolutionary system for which a family of trajectories consists of all weak solutions (instead of all Leray--Hopf weak solutions as in \cite{Cheskidov}) of the critical convective Brinkman--Forchheimer equations in $X$. More precisely, we define
\begin{align*} \label{evolutionarycbf}
\mathcal{E}([T, \infty)) := \{&u(\cdot) : u(\cdot) \mbox{ is a weak solution on } [T, \infty) \\ \nonumber
&\mbox{ and } u(t) \in X \ \forall \, t \in [T, \infty)\}, \quad T \in \mathbb{R}, \\ \nonumber
\mathcal{E}((-\infty, \infty)) := \{&u(\cdot) : u(\cdot) \mbox{ is a weak solution on } (-\infty, \infty) \\ \nonumber
&\mbox{ and } u(t) \in X \ \forall \, t \in (-\infty, \infty)\}.
\end{align*}
Clearly, the properties $1$--$4$ of an evolutionary system $\mathcal{E}$ hold. Therefore, thanks to Theorem \ref{weakglobalattractor}, the weak global attractor $\mathcal{A}_w$ exists for this evolutionary system. Additionally, we can prove the following theorem.
\begin{theorem}
The weak global attractor $\mathcal{A}_w$ for the evolutionary system $\mathcal{E}$ of the critical CBF equations is a strongly compact strong global attractor $\mathcal{A}_s$.
\end{theorem}

\begin{proof}
Since every complete trajectory of the evolutionary system $\mathcal{E}$ for the critical CBF equations is strongly continuous, due to Theorem \ref{asymptoticcompactness} and Theorem \ref{globalattractor}, it is enough to prove that $\mathcal{E}$ satisfies the assumptions \ref{Aone}, \ref{Atwo}, and \ref{Athree}.

First note that $\mathcal{E}([0, \infty)) \subset C([0, \infty); H_w)$ by the definition of weak solutions, see $(\ref{weakctn})$. Now take any sequence $u_n \in \mathcal{E}([0, \infty))$ for $n = 1, 2, \dots$. Thanks to classical estimates for Leray--Hopf weak solutions of the NSE (Lemma $8.5$ in \cite{Cheskidov}, for more details see \cite{ConstantinFoias}), which apply also to the CBF equations, there exists a subsequence, still denoted by $u_n$, that converges to some $u^1 \in \mathcal{E}([0, \infty))$ in $C([0, 1]; H_w)$ as $n \to \infty$. Passing to a subsequence  and dropping a subindex once more, we obtain that $u_n \to u^2$ in $C([0, 2]; H_w)$ as $n \to \infty$ for some $u^2 \in~\mathcal{E}([0, \infty))$. Note that $u^1(t) = u^2(t)$ on $[0, 1]$. Continuing this diagonalisation procedure, we obtain a~subsequence $\lrb{u_{n_j}} \subset \lrb{u_n}$ that converges to some $u \in \mathcal{E}([0, \infty))$ in $C([0, \infty); H_w)$ as $n_j \to \infty.$ Therefore, \ref{Aone} holds.

The energy inequality \ref{Atwo} follows immediately from the energy equality $(\ref{equality})$. Note also that to prove \ref{Atwo} it is enough to have only the strong energy inequality $(\ref{energyineq})$.

Let now $u_n, u \in \mathcal{E}([0, T])$ be such that $u_n \to u$ in $C([0, \infty); H_w)$ as $n \to \infty$ for some $T > 0$. Classical estimates for the NSE (see e.g. \cite{ConstantinFoias} or \cite{RRSbook}), which hold as well for the CBF equations, imply that the sequence $\lrb{\partial_tu_n}$ is bounded in $L^{4/3}(0, T; V')$, where $V'$ is the dual space of $V$. Since the sequence $\lrb{u_n}$ is bounded in $L^2(0, T; V)$, by the Aubin--Lions Lemma, there exists a subsequence $\lrb{u_{n_j}} \subset \lrb{u_n}$, such that
$$ \calkat{\norm{u_{n_j}(t) - u(t)}^2} \to 0 \quad \mbox{as} \quad {n_j} \to \infty. $$
In particular, $\norm{u_{n_j}(t)} \to \norm{u(t)}$ as ${n_j} \to \infty$ for a.e. $t \in [0, T]$, i.e.\ \ref{Athree} holds.
\end{proof}

Finally, we note that all the other results from \cite{Cheskidov} apply to the three-dimensional convective Brinkman--Forchheimer equations $(\ref{alphazero})$ as well. For instance, \textit{the~trajectory attractor} $\mathcal{U}$ exists for the critical CBF equations, and uniformly attracts $\mathcal{E}([0, \infty))$ in $L_{\mathrm{loc}}^{\infty}((0, \infty); H)$.

\bibliography{biblio}
\bibliographystyle{acm}

\end{document}